\documentclass{amsproc}

\usepackage{amsmath}
\usepackage{amssymb}
\usepackage{amsthm}
\usepackage{amscd}
\usepackage{amsfonts}
\usepackage{fancyhdr}
\usepackage{enumerate}

\begin{document}

\title[The cluster value problem in continuous function]{The cluster value problem in spaces of continuous functions}


%    author one information
\author{W.~B.~Johnson}
\address{Department of Mathematics, Texas A\&M University, College Station, TX 77843}
\curraddr{}
\email{johnson@math.tamu.edu}
\thanks{The authors were supported in part by NSF DMS 10-01321.}

%    author two information
\author{S.~Ortega Castillo}
\address{Department of Mathematics, Texas A\&M University, College Station, TX 77843}
\curraddr{CIMAT A.~C., Guanajuato, Guanajuato, M\'exico 36240}
\email{sofia.ortega@cimat.mx}
\thanks{}


\newtheorem{thm}{Theorem}[section]

\newtheorem{df}[thm]{Definition}
\newtheorem{ex}[thm]{Example}
\newtheorem{lm}[thm]{Lemma}
\newtheorem{pr}[thm]{Proposition}
\newtheorem{co}[thm]{Corollary}
\newtheorem{re}[thm]{Remark}
\newtheorem{note}[thm]{Note}
\newtheorem{claim}[thm]{Claim}
\newtheorem{problem}[thm]{Problem}

\numberwithin{equation}{section} 

%\def\R{{\mathbb R}}

%\def\E{\mathbb{E}}
%\def\calF{{\cal F}}
%\def\N{\mathbb{N}}
%\def\calN{{\cal N}}
%\def\calH{{\cal H}}
%\def\n{\nu}
%\def\a{\alpha}
%\def\d{\delta}
%\def\t{\theta}
%\def\e{\varepsilon}
%\def\t{\theta}
%\def\g{\gamma}
%\def\G{\Gamma}
%\def\b{\beta}
%\def\pf{ \noindent {\bf Proof: \  }}

%\newcommand{\qed}{\hfill\vrule height6pt
%width6pt depth0pt}

%\def\endpf{\qed \medskip} \def\colon{{:}\;}
%\setcounter{footnote}{0}

%\renewcommand{\qed}{\hfill\vrule height6pt  width6pt depth0pt}

%opening
%\title{\textbf{The cluster value problem in spaces of continuous functions.\footnote{Primary subject classification: %Several complex variables and analytic spaces.}\footnote{Secondary subject classification: Functional analysis.}}}
%\date{October 31, 2014}
%\author{W.~B.~Johnson\footnote{Supported in part by NSF DMS 10-01321 } \text{ }and S.~Ortega %Castillo\footnotemark[\value{footnote}]}

\date{}

\begin{abstract}
We study the cluster value problem for certain Banach algebras of holomorphic functions defined on the unit ball of a complex Banach space $X$. The main results are for spaces of the form $X=C(K).$
\end{abstract}

\maketitle

\section{Preliminaries.}

A cluster value problem for a complex Banach space $X$ is a weak version of the corona problem for the open unit ball $B$ of $X$, which is a long-standing open problem in complex analysis when $X$ has dimension at least $2$. Instead of studying when $B$ is dense in the spectrum of a uniform algebra $H$ of bounded analytic functions on $B$ in the weak topology induced by $H$(corona problem), the cluster value problem investigates the following situation:

\smallskip

Let $\bar{B}^{**}$ be the closed unit ball of the bidual $X^{**}$, and let $M_H$ be the spectrum (i.e. maximal ideal space) of a uniform algebra $H$ of norm continuous functions on $B$ with $H \supset X^*.$ Then $\pi: M_{H} \to \bar{B}^{**},$ given by $\pi(\tau)=\tau|_{X^*}$ for $\tau \in M_{H},$ is surjective (as a consequence of the results in Chapter 2 and 5 of \cite{G}). For each $x^{**} \in \bar{B}^{**},$ $M_{x^{**}}(B)=\pi^{-1}(x^{**})$ is called the fiber of $M_{H}$ over $x^{**}.$ Aron, Carando, Gamelin, Lasalle and Maestre observed in \cite{ACGLM} that for every $x^{**} \in \bar{B}^{**}$ we have the inclusion

\begin{equation}
Cl_{B}(f,x^{**})\subset \widehat{f}(M_{x^{**}}(B)), \; \forall f \in H,
\end{equation}

where $Cl_{B}(f,x^{**}),$ \textit{the cluster set of $f$ at $x^{**}$}, stands for the set of all limits of values of $f$ along nets in $B$ converging weak-star to $x^{**}$, while $\widehat{f}$ represents the Gelfand transform of $f.$  There they formulated the cluster value problem for $H:$ for which Banach spaces $X$ is there equality in (1) for all $x^{**} \in \bar{B}^{**}$? When there is equality in (1) for a certain $x^{**} \in \bar{B}^{**},$ we say $X$ satisfies \textit{the cluster value theorem for $H$ at $x^{**}.$} 

\smallskip

As was pointed out in \cite{ACGLM}, it is easy to check that the cluster value theorem for $H$ at all points in $\bar{B}^{**}$ is indeed weaker than the corona problem for $B$ and $H:$ Given $x^{**} \in \bar{B}^{**},$ if $\tau \in M_{x^{**}}(B)$ were the weak-star limit of the net $(x_{\alpha}) \subset B,$ then $\lim_{\alpha} f(x_{\alpha})=\widehat{f}(\tau)$ for all $f \in H,$ and in particular $\lim_{\alpha} x^*(x_{\alpha})=\widehat{x^*}(\tau)=x^*(x^{**})$ for all $x^* \in X^*,$ i.e. $x^{**}$ would be the weak-star limit of $(x_{\alpha}),$ and so $\widehat{f}(\tau) \in Cl_{B}(f,x^{**}).$

\smallskip

In an effort to research the corona problem, we investigate conditions that guarantee the simpler cluster value theorem for a Banach algebra of analytic functions defined on the unit ball of a complex Banach space $X.$ In particular, we generalize some of the results in \cite{ACGLM}. Our main results are for the spaces of the form $X=C(K),$ including a translation result of a cluster value problem, from any point in $B_{C(K)}$ to the origin.

%\smallskip

%We thank Richard Aron and Manuel Maestre for their communications.

\section{Cluster value problem in finite-codimensional subspaces.}

In \cite{ACGLM}, the authors obtain a cluster value theorem at the origin for Banach spaces with shrinking 1-unconditional bases for the algebra $H=A_u(B)$ of bounded analytic functions on $B$ that are also uniformly continuous. Slight modifications of their arguments in Section 3 of \cite{ACGLM} yield the following:

\begin{pr}\label{pr2.1}
Let $S$ be a finite rank operator on $X,$ so that $P=I-S$ has norm one. If $\phi \in M_0(B),$ then $\hat{f}(\phi)=\widehat{f\circ P}(\phi),$ for all $f \in A_u(B)$.  
\end{pr}

\begin{pr}\label{pr2.2}
Suppose that for each finite dimensional subspace $E$ of $X^*$ and $\epsilon>0$ there exists a finite rank operator $S$ on $X$ so that $||(I-S^*)|_E||<\epsilon$ and $||I-S||=1.$ Then the cluster value theorem holds for $A_u(B)$ at $0.$
\end{pr}
\begin{proof}
Suppose that $0 \notin Cl_B(f,0).$ We must show that $0 \notin \hat{f}(M_0).$ Since $0 \notin Cl_B(f,0),$ there exists $\delta>0$ and a weak neighborhood $U$ of $0$ in $X$ such that $|f|\geq \delta$ on $U\cap B.$ Without loss of generality we may assume $U=\cap_{i=1}^n \{ x \in X: \; \; |x_i^*|<\epsilon_0 \}$ for some $x_1^*, \cdots, x_n^* \in B_{X^*}$ and $\epsilon_0>0.$ Let $E=\text{span}\{ x_1^*,\cdots, x_n^*\}$ and let $S$ be as in the statement for $\epsilon=\epsilon_0.$ Then $|f \circ (I-S)|\geq \delta$ on $B,$ because for every $x \in B$ we have that $(I-S)x \in U,$ indeed:
$$|<x_i^*, (I-S)x>|=|<(I-S^*)x_i^*,x>|<\epsilon_0, \text{ for } i=1,\cdots, n.$$
Consequently $f\circ(I-S)$ is invertible in $A_u(B).$ Hence $\widehat{f\circ (I-S)}\neq 0$ on the fiber of the spectrum of $A_u(B)$ over $0.$ From the preceding lemma we then obtain $\hat{f}\neq 0$ on $M_0,$ that is $0 \notin \hat{f}(M_0).$
\end{proof} 

Since Proposition \ref{pr2.2} builds on Proposition \ref{pr2.1}, one naturally wonders if Proposition \ref{pr2.1} can be extended to the larger algebra $H^{\infty}(B)$ of all bounded analytic functions on $B$. The answer is no in general, as shown by the following example of Aron.

\begin{ex}\label{ex2.3}
There exists a finite rank operator $S$ on $\ell_2$ so that $P=I-S$ has norm one, and there exist $\phi \in M_0(B_{\ell_2})$ as well as $f \in H^{\infty}(B_{\ell_2})$ so that $\hat{f}(\phi)\neq \widehat{f\circ P}(\phi)$.
\end{ex}
\begin{proof}
Let $S:\ell_2 \to \ell_2$ be given by $S(x)=(x_1,0,0,\cdots).$

\medskip

Clearly $S$ is a finite rank operator and $P=I-S$ has norm one.

\medskip

Let $(r_j)$ and $(s_j)$ be sequences of positive real numbers, such that $(r_j)\downarrow 0$ and $(s_j)\uparrow 1$ in such a way that each $r_j^2+s_j^2<1$ and $r_j^2+s_j^2 \to 1^{-}.$ For each $j=1,2,3,\cdots,$ let $\delta_{r_j e_1+s_j e_j}$ be the usual point evaluation homomorphism from $H^{\infty}(B_{\ell_2}) \to \mathbb{C}.$ Let $\phi: H^{\infty}(B_{\ell_2}) \to \mathbb{C}$ be an accumulation point of $\{ \delta_{r_j e_1+s_j e_j}\}$ in the spectrum of $H^{\infty}(B_{\ell_2}).$ Let $f: B_{\ell_2} \to \mathbb{C}$ be the $H^{\infty}$ function given by
$$f(x)=\frac{x_1}{\sqrt{1-\sum_{j=2}^{\infty} x_j^2 }},$$
where the square root is taken with respect to the usual logarithm branch. Then $\phi(f)=1.$ However $\phi(f\circ P)=0$ since $f\circ P\equiv 0.$

\end{proof}

When a Banach space has a shrinking reverse monotone finite dimensional decomposition (FDD), that is, a shrinking FDD so that the natural projections are at distance one from the identity operator, we have that the condition in Proposition \ref{pr2.2} holds, and therefore we obtain a cluster value theorem:

\begin{co}\label{co2.4}
 If $X$ is a Banach space with a shrinking reverse monotone FDD, then the cluster value theorem holds for $A_u(B)$ at 0.
\end{co}

The operators $P$ considered in Propositions \ref{pr2.1} and \ref{pr2.2} have finite-codimensional rank, which suggests that the cluster value problem at the origin of a Banach space can be studied by considering the same problem in its finite-codimensional subspaces. We established the following relationship with the help of Aron and Maestre:

\begin{pr}\label{pr2.5}
If $Y$ is a closed finite-codimensional subspace of $X$ and $f \in A_u(B),$ then $Cl_{B}(f,0)=Cl_{B_Y}(f|_Y,0),$ where $B_Y$ is the unit ball of $Y.$
\end{pr}
\begin{proof}

$A_u(B)$ coincides with the uniform limits on $\bar{B}$ of continuous polynomials on $X$ (see Theorem 7.13 in \cite{M} and p.~56 in \cite{ACG}), where polynomials are finite linear combinations of symmetric multi-linear mappings (of possibly distinct degrees) restricted to the diagonal. Thus, by passing to the uniform limit on $\bar{B}$, we may assume $f$ is an $m$-homogeneous polynomial, with associated symmetric $m$-linear functional $F.$ Let $(x_{\alpha})$ be a weakly null net in $B$ such that $f(x_{\alpha}) \to \lambda.$ 

\medskip

Each $x_{\alpha}$ can be written uniquely as $y_{\alpha}+u_{\alpha},$ where $y_{\alpha} \in Y$ and $u_{\alpha}$ is from a fixed finite dimensional complement of $Y$ in $X.$ Then 
\begin{align*} 
&f(x_{\alpha})\\
=&F(x_{\alpha},\cdots, x_{\alpha})\\
=&f(y_{\alpha})+m F(y_{\alpha},\cdots, y_{\alpha}, u_{\alpha})+[m(m-1)/2] F(x_{\alpha},\cdots, x_{\alpha},u_{\alpha},u_{\alpha})+\cdots+f(u_{\alpha}).
\end{align*}

Now, since $(x_{\alpha})$ is weakly null, the same holds for $(y_{\alpha})$ and $(u_{\alpha}).$ However, since $(u_{\alpha})$ belongs to a finite dimensional space, it follows that $|| u_{\alpha} || \to 0.$ Thus $F(y_{\alpha} \cdots, y_{\alpha}, u_{\alpha}),$ $F(y_{\alpha},\cdots, y_{\alpha},u_{\alpha},u_{\alpha}), \cdots, f(u_{\alpha})$ all go to $0.$ Thus $f(y_{\alpha}) \to \lambda.$ Finally, since each $||y_{\alpha}||\leq ||x_{\alpha}||+||-u_{\alpha}||<1+||u_{\alpha}||$, then by defining $t_{\alpha}=\frac{1}{1+||u_{\alpha}||}$ we get that $|| t_{\alpha} y_{\alpha} || < 1$ for all $\alpha$ and $t_{\alpha} \to 1,$
 and consequently, $\lim f(t_{\alpha} y_{\alpha})= \lim t_{\alpha}^m f(y_{\alpha})=\lambda.$ Hence $\lambda \in Cl_{B_Y}(f|_Y,0).$
%Finally, since $\limsup ||y_{\alpha}||$ $\leq 1,$ we can take a sequence of scalars $(t_{\alpha})$ such that $|| t_{\alpha} %y_{\alpha} || < 1$ for all $\alpha$ and $t_{\alpha} \to 1,$
 %and consequently, $\lim f(t_{\alpha} y_{\alpha})= \lim t_{\alpha}^m f(y_{\alpha})=\lambda.$ Hence $\lambda \in %Cl_{B_Y}(f|_Y,0).$

\end{proof}  

As a consequence we obtain that the cluster sets of an element $f$ of $A_u(B)$ at 0 can be described in terms of the Gelfand transforms of $f|_{B_Y}$ as $Y$ ranges over finite-codimensional subspaces of $X$:

\begin{pr}\label{pr2.6}
For every Banach space $X,$ $$Cl_{B}(f,0)=\bigcap_{Y\subset X,\dim(X/Y)<\infty} \widehat{f|_{B_Y}} (M_0(B_Y)) , \; \forall f \in A_u(B).$$
\end{pr}
\begin{proof}
From Proposition \ref{pr2.5} and the inclusion in (1), for every finite-codimensional subspace $Y$ of $X,$
$$Cl_B(f,0)=Cl_{B_Y}(f|_{B_Y},0) \subset \widehat{f|_{B_Y}} (M_0(B_Y)).$$

\medskip

For the reverse inclusion, suppose $0 \notin Cl_B(f,0).$ Then there are $\epsilon >0$ and a weak neighborhood $U$ of $0$ such that $|f|>\epsilon$ on $U \cap B.$ $U$ contains a closed finite-codimensional subspace $Y_0$ of $X,$ so $|f|_{B_{Y_0}}|>\epsilon.$ Hence $\widehat{f|_{B_{Y_0}}}$ is invertible, which implies that $0 \notin \widehat{f|_{B_{Y_0}}} (M_0(B_{Y_0})).$
\end{proof}

Going back to Proposition \ref{pr2.2}, we see that having the cluster value property at 0 only requires the existence of a certain type of finite rank operators at distance one from the identity operator. However simple this condition may seem, it is impossible in the case of the Banach space $c$ of continuous functions on $\omega,$ also seen as the subspace of $l^{\infty}$ of convergent sequences:

\begin{ex}\label{ex2.7}
Let $L \in B_{c^*}$ be given by $$L((c_n)_n)=\lim_{n \to \infty} c_n.$$ If $S: c \to c$ is a finite rank operator with $||(S^*-I_{c^*})L||<\epsilon,$ then $||S-I_c|| \geq 2-\epsilon$.
\end{ex}
\begin{proof}

For each $k \in \mathbb{N},$ consider $L_k \in B_{c^*}$ given by $$L_k((c_n)_n)=(\lim_{n \to \infty}c_n-c_k)/2.$$

\medskip

Let us show that $||S^*(L_k)|| \to 0$ as $k \to \infty.$ For every $x \in B_c,$ $S^*(L_k)x=L_k(Sx) \to 0$ as $k \to \infty.$ Moreover, since $S$ has finite rank, $\{Sx: \; x \in B_c\}$ is pre-compact. Thus $S^*L_k=L_k\circ S$ converges to zero uniformly on $B_c,$ i.e. $||S^*L_k||\to 0$ as $k \to \infty.$

\medskip

Now note that $||L-2L_k||=1$ for each $k$, so
$$||S^*-I_{c^*}||\geq ||(S^*-I_{c^*})(L-2L_k)||\geq||2L_k-2 \cdot S^*(L_k)||-\epsilon \geq 2-\epsilon-2||S^*(L_k)||.$$
Since $S^*(L_k) \to 0,$ then $||S-I_c||=||S^*-I_{c^*}||\geq 2-\epsilon.$ 

%Let us first prove $S(e_n) \to 0$ as $n \to \infty.$ Since $S|_{c_0}:c_0 \to c$ is a finite rank operator, we can find $\overrightarrow{a^1}, \cdots, \overrightarrow{a^m} \in c_0^*=l^1$ and $\overrightarrow{b^1}, \cdots, \overrightarrow{b^m} \in c$ such that $S|_{c_0}=\sum_{j=1}^{m} \overrightarrow{a^j} \otimes \overrightarrow{b^j}.$ Then $S(e_n)=\sum_{j=1}^{m} a^j_n \cdot \overrightarrow{b^j},$ so $||S(e_n)||\leq\sum_{j=1}^{m} |a^j_n| \cdot ||\overrightarrow{b^j}||.$ Thus $||S(e_n)|| \to 0.$

\medskip

%Now note that $||1-2e_n||=1$ for each $n$, so
%$$||S-I||\geq ||S(\overrightarrow{1}-2e_n)-(\overrightarrow{1}-2e_n)||\geq||2e_n-2 \cdot S(e_n)||-\epsilon \geq 2-\epsilon-||S(e_n)||.$$
%Since $S(e_n) \to 0,$ then $||S-I||\geq 2-\epsilon.$ 
\end{proof}

The reader may check that the condition is also impossible for $L_p,$ $1\leq p \neq 2 < \infty.$

\bigskip

However, note that since $c_0$ is one-codimensional in $c,$ Proposition \ref{pr2.5} implies that for all $f \in A_u(B_c),$
$$Cl_{B_c}(f,0)=Cl_{B_{c_0}}(f|_{B_{c_0}},0).$$

\smallskip

Also, Propositions 1.59 and 2.8 of \cite{D} imply that all functions in $A_u(B_{c_0})$ can be uniformly approximated on $B$ by polynomials in the functions in $X^*,$ which in turn implies that each fiber at $x \in \bar{B}^{**}$ consists only of $x,$ so the cluster value theorem for $A_u(B_{c_0})$ holds, and in particular
$$Cl_{B_{c_0}}(f|_{B_{c_0}},0)=\widehat{f|_{B_{c_0}}} (M_0(B_{c_0})), \; \; \forall f \in A_u(B_{c}).$$

Hence we are left to compare $\widehat{f|_{B_{c_0}}} (M_0(B_{c_0}))$ with $\widehat{f} (M_0(B_{c}))$ for $f \in A_u(B_c).$ Note that an inclusion is evident:

\begin{pr}\label{pr2.8}
For a Banach space $X$ and $Y$ a subspace of $X,$ 
$$\widehat{f|_{B_Y}} (M_0(B_{Y})) \subset \widehat{f} (M_0(B)), \; \; \forall f \in A_u(B).$$
\end{pr}
\begin{proof}
Let $f \in A_u(B)$ and $\tau \in M_0(B_Y).$ Since $\phi_1: A_u(B) \to A_u(B_Y)$ given by $\phi(g)=g|_Y$ for all $g \in A_u(B)$ is a continuous homomorphism that maps $A(B)$ into $A(B_Y),$ the mapping $\tilde{\tau}: A_u(B) \to \mathbb{C}$ given by $\tilde{\tau}(g)=\tau(g|_Y)$ for all $g \in A_u(B)$ is in the fiber $M_0(B).$ Moreover, $$\widehat{f|_Y}(\tau)=\hat{f}(\tilde{\tau}).$$
\end{proof}

The reverse inclusion is unclear. However, the space $c$ also has the property of being isomorphic to $c_0,$ which implies, as we will see, that $c$ has the cluster value property too.

\bigskip

Let $P(X)$ denote the continuous polynomials on $X,$ $P_f(X)$ the polynomials in the functions of $X^*$ (known as finite type polynomials), and $A(B_X)$ the uniform algebra of uniform limits of elements in $P_f(X).$

\begin{lm}\label{lm2.9}
Let $X$ be a Banach space so that $A_u(B_X)=A(B_X).$ If the Banach space $Y$ is isomorphic to $X,$ then also $A_u(B_Y)=A(B_Y)$.
\end{lm}
\begin{proof}

Let $T: Y \to X$ be the Banach space isomorphism between $Y$ and $X.$

\smallskip

Let $f \in A_u(B_Y).$ Then there exist a sequence of polynomials $P_n \in \mathcal{P}(Y)$ such that $||P_n-f||_{B_Y} \leq \frac{1}{n}, \; \; \forall n \in \mathbb{N}.$ 

\smallskip

For each $n \in \mathbb{N},$ $P_n \circ T^{-1} \in \mathcal{P}(X),$ so there exists a polynomial $Q_n \in \mathcal{P}_f(X)$ such that $||P_n \circ T^{-1}-Q_n||_{B_X}<\frac{1}{n \cdot ||T||}$, and consequently $||P_n-Q_n \circ T||_{B_Y}<\frac{1}{n},$ where $Q_n \circ T \in \mathcal{P}_f(Y).$ 

\smallskip

Consequently, the sequence of polynomials $Q_n \circ T \in \mathcal{P}_f(Y)$ converges to $f$ uniformly on $B_Y,$ so $f \in A(B_Y).$

\end{proof}

\begin{co}\label{co2.10}
The Banach space $c$ satisfies the cluster value theorem for $A_u(B_c)$ at all points in $\overline{B_c}^{**}.$
\end{co}

\section{Cluster value problem in $C(K)\ncong c.$}

Bessaga and Pe\l czy\'nski proved in \cite{BP} that, when $\alpha \geq \omega^{\omega}$ is a countable ordinal, $C(\alpha)$ is not isomorphic to $c=C(\omega).$ Therefore we no longer can use Lemma \ref{lm2.9} to obtain a cluster value theorem on such spaces of continuous functions. 

\medskip

Nevertheless, for $\alpha$ a countable ordinal, the intervals $[1, \alpha]$ are always compact, Hausdorff and dispersed (they contain no perfect non-void subset). The compact, Hausdorff and dispersed sets $K$ satisfy, from the Main Theorem in \cite{PS}, that $X=C(K)$ contains no isomorphic copy of $l_1.$ Moreover, from Theorem 5.4.5 in \cite{AK}, $X=C(K)$ has the Dunford-Pettis property. Therefore, for dispersed $K,$ the continuous polynomials on $X=C(K)$ are weakly (uniformly) continuous on bounded sets by Corollary 2.37 in \cite{D}.

\medskip

Moreover, since $X^*=l_1(K)$ has the approximation property,  Proposition 2.8 in \cite{D} and the conclusion in the former paragraph now yield that all continuous polynomials on $X$ can be uniformly approximated, on bounded sets, by polynomials of finite type. Thus the elements of $A_u(B)$ can be approximated, uniformly on $B,$ by polynomials of finite type. Hence $A_u(B)=A(B),$ so each fiber at $x \in \bar{B}^{**}$ is the singleton $\{x\},$ and then $X$ satisfies the cluster value theorem for the algebra $A_u(B).$ 

\bigskip

We now consider the cluster value problem on $X$ for the algebra of all bounded analytic functions $H^{\infty}(B).$ Following the line of proof of Theorem 5.1 in \cite{ACGLM}, we still get a cluster value theorem:

\begin{thm}\label{thm3.1}
If $X$ is the Banach space $C(K),$ for $K$ compact, Hausdorff and dispersed, then the cluster value theorem holds for $H^{\infty}(B)$ at every $x \in \bar{B}^{**}$.
\end{thm}
\begin{proof}
Fix $f \in H^{\infty}(B)$ and $w=(w_t)_{t \in K} \in \bar{B}^{**}$ (where $C(K)^{**}=l_{\infty}(K)$). Suppose $0 \notin Cl_B(f, w).$ It suffices to show that $0 \notin \hat{f}(M_w).$ 

\medskip

Since 0 is not a cluster value of $f$ at $w,$ there exists a weak-star neighborhood $U$ of $w$ such that $0 \notin \overline{f(U \cap B)},$ where $$U \cap B\supset \cap_{i=1}^{n} \{ z \in B: | < (z-w), x_i^* > | < \epsilon \},$$ for some $\epsilon > 0$ and $x_1^*, \cdots, x_n^* \in X^*=l_1(K).$

\medskip

We have that $x_i^*=(x_i^*(t))_{t \in K}$ has countably many nonzero coordinates $\{x_i^*(t)\}_{t \in F_i}$ for $i=1,\cdots,n$. Then, $$U \cap B\supset \cap_{i=1}^{n} \{ z \in B: |\sum_{t \in K} (z_{t}-w_{t})x_i^*(t)| < \epsilon \},$$

\medskip

and there is a finite set $F\subset \cup_{i=1}^n F_i$ so that $\sum_{t \notin F} |x_i^*(t)|<\epsilon/4,$ for $i=1, \cdots, n.$ Then, $$U \cap B\supset \cap_{t \in F} \{ z \in B: |z_{t}-w_{t}| < \delta \},$$ where $$\delta=\min_{1\leq i \leq n, t \in F } \frac{\epsilon}{(2|F|)|x_i^*(t)|}.$$
  
\medskip

In summary, there exist $c > 0,$ $\delta > 0$ and a finite set $F\subset K$ such that if $z \in B$ satisfies $|z_{t}-w_{t}| < \delta$ for $t \in F$ then $|f(z)| \geq c.$ Relabel the indices in $F$ as $t_1, \cdots, t_m,$ where $m=|F|.$ Then proceed as in the proof of Theorem 5.1 in \cite{ACGLM}:

\medskip

For $0\leq k \leq m-1,$ define $U_k=\{z \in B: |z_{t_j}-w_{t_j}|<\delta, k+1 \leq j \leq m \},$ and set $U_m=B.$ Note that $1/f$ is bounded and analytic on $U_0.$

\medskip

We claim that for each $k,$ $1\leq k\leq m,$ there are functions $g_k$ and $h_{k,j},$ $1\leq j \leq k,$ in $H^{\infty}(U_k)$ that satisfy
\begin{equation}
f(z)g_k(z)=1+(z_{t_1}-w_{t_1})h_{k1}(z)+\cdots+(z_{t_k}-w_{t_k})h_{kk}(z), \; \; z \in U_k.
\end{equation}

\smallskip

Once this claim is established, the proof is easily completed as follows. The functions $g_m$ and $h_{mj}$ belong to $H^{\infty}(B)$ and satisfy
$$\widehat{f} \widehat{g_m}=\widehat{1}+\sum_{j=1}^m \widehat{(z_{t_j}-w_{t_j})} \widehat{h_{mj}}.$$

\smallskip

Since each  $\widehat{z_{t_j}}-w_{t_j}$ vanishes on $M_w$ (by the definition of $M_w$), we obtain $\widehat{f}\widehat{g_m}=1$ on $M_w,$ and consequently $\widehat{f}$ does not vanish on $M_w,$ as required.

\medskip

Just as in \cite{ACGLM}, the claim is established by induction on $k.$ The first step, the construction of $g_1$ and $h_{11},$ is as follows. We regard $1/f((z_t)_{t \in K})$ as a bounded analytic function of $z_{t_1}$ for $|z_{t_1}|<1$ and $|z_{t_1}-w_{t_1}|<\delta,$ with $z_t,$ $t \in K-\{t_1\},$ as analytic parameters in the range $|z_t|<1$ for $t \in K-\{t_1\},$ and $|z_{t_j}-w_{t_j}|<\delta$ for $2 \leq j \leq m.$ According to lemma 5.3 in \cite{ACGLM}, we can express
$$\frac{1}{f(z)}=g_1(z)+(z_{t_1}-w_{t_1})h(z), \; \; z \in U_0,$$
where $g_1 \in H^{\infty}(U_1)$ and $h \in H^{\infty}(U_0)$. We set
$$h_{11}(z)=[f(z)g_1(z)-1]/(z_{t_1}-w_{t_1}), \; \; z \in U_1,$$
so that (2) is valid for $k=1.$ Note that $h_{11}=-hf$ on $U_0.$ Consequently $h_{11}$ is bounded and analytic on $U_0.$ The defining formula then shows that $h_{11}$ is analytic on all of $U_1,$ and since $|z_{t_1}-w_{t_1}|\geq \delta$ on $U_1-U_0,$ $h_{11}$ is bounded on $U_1.$

\medskip

Now suppose that $2\leq k \leq m,$ and that there are functions $g_{k-1}$ and $h_{k-1,j} \; \; (1\leq j \leq k-1)$ that satisfy (2) and are appropriately analytic. We apply lemma 5.3 in \cite{ACGLM} to these as functions of $z_{t_k},$ with the other variables regarded as analytic parameters, to obtain decompositions 
$$g_{k-1}(z)=g_k(z)+(z_{t_k}-w_{t_k})G_k(z)$$
and
$$h_{k-1,j}(z)=h_{k,j}(z)+(z_{t_k}-w_{t_k})H_{k,j}(z), \; \; 1\leq j \leq m-1,$$
where $g_k$ and the $h_{kj}$'s are in $H^{\infty}(U_{k}),$ and $G_k$ and the $H_{kj}$'s are in $H^{\infty} (U_{k-1}).$ From the identity (2), with $k$ replaced with $k-1,$ we obtain 
$$fg_k=1+\sum_{j=1}^{k-1}(z_{t_j}-w_{t_j})h_{kj}+(z_{t_k}-w_{t_k})[-fG_k+\sum_{j=1}^{k-1}(z_{t_j}-w_{t_j})H_{kj}]$$ 
on $U_{k-1}.$ We define
$$h_{kk}=[fg_k-1-\sum_{j=1}^{k-1}(z_{t_j}-w_{t_j})h_{kj}]/(z_{t_k}-w_{t_k}), \; \; z \in U_k.$$
Then (2) is valid. On $U_{k-1}$ we have
$$h_{kk}=-fG_k+\sum_{j=1}^{k-1}(z_{t_j}-w_{t_j})H_{kj},$$
so that $h_{kk}$ is bounded and analytic on $U_{k-1}.$ Since $|z_{t_k}-w_{t_k}|\geq \delta$ on $U_k-U_{k-1},$ we see from the defining formula that $h_{kk} \in H^{\infty}(U_k).$ This establishes the induction step, and the proof is complete.

\end{proof}

We do not know the answer to the cluster value problem for other spaces $C(K).$ 

\bigskip

Consider the following cluster value problem: Given $f_0^{**} \in \overline{B}^{**},$ the cluster value problem for $H^{\infty}(B)$ over $A_u(B)$ at $f_0^{**}$ asks whether for all $\psi \in H^{\infty}(B)$ and $\tau \in \mathcal{M}_{f_0^{**}}(B)$ ($\mathcal{M}_{f_0^{**}}(B)=\pi^{-1}(\delta_{f_0^{**}})$ for the restriction map $\pi: M_{H^{\infty}(B)} \to  M_{A_u(B)}$), can we find a net $(f_{\alpha})\subset B$ such that $\psi(f_{\alpha}) \to \tau(\psi)$ and $f_{\alpha}$ converges to $f_0^{**}$ in the polynomial-star topology, i.e. the smallest topology that makes every extension of a polynomial on $X$ to $X^{**}$ continuous (that we denote by $\tau(\psi) \in \mathsf{Cl}_B(\psi, f_0^{**})$)? As before, clearly $\mathsf{Cl}_B(\psi, f_0^{**})\subset \widehat{\psi}(\mathcal{M}_{f_0^{**}}(B)), \; \forall \psi \in H^{\infty}(B)$.

%Consider the following cluster value problem: Given $f_0^{**} \in \overline{B}^{**},$ the cluster value problem for %$H^{\infty}(B)$ over $A_u(B)$ at $f_0^{**}$ asks whether for all $\psi \in H^{\infty}(B)$ and $\tau \in M_{H^{\infty}(B)}$ %such that $\tau|_{A_u(B)}=f_0^{**}$ (that we denote by $\tau \in \mathcal{M}_{f_0^{**}}(B)$), can we find a net %$(f_{\alpha})\subset B$ such that $\psi(f_{\alpha}) \to \tau(\psi)$ and $f_{\alpha}$ converges to $f_0^{**}$ in the %polynomial-star topology, as defined in p.~200 in \cite{G} (that we denote by $\tau(\psi) \in \mathsf{Cl}_B(\psi, %f_0^{**})$)?

\bigskip

The cluster value problem for $H^{\infty}(B)$ over $A_u(B)$ coincides with the cluster value problem for $H^{\infty}(B)$ when $A_u(B)=A(B)$. Thus when $K$ is compact, Hausdorff and dispersed, we have a positive answer to the previous cluster value problem for such $C(K)$ spaces. 

\bigskip

The previous problem seems to be highly nontrivial. Since for every infinite compact Hausdorff space $K$, $C(K)$ contains a subspace $Y$ isometric to $c_0$ (Proposition 4.3.11 in \cite{AK}), the fiber $\mathcal{M}_0(B_{C(K)})$ is huge (and from Lemma \ref{lm3.3}, also each fiber $\mathcal{M}_{f_0}(B_{C(K)})$ for $f_0 \in B_{C(K)}$). Indeed, according to Theorem 6.6 in \cite{CGJ}, there is a family of distinct characters $\{\tau_{\alpha}\}_{\alpha \in B_{\ell_{\infty}}}$, such that each $\tau_{\alpha}:H^{\infty}(B_Y)\to\mathbb{C}$ satisfies $\delta_0=\tau_{\alpha}|_{A(B_Y)}=\tau_{\alpha}|_{A_u(B_Y)}$ (because $Y$ is isometric to $c_0$, so $A(B_Y)=A_u(B_Y)$). Hence  $\{\tau_{\alpha}\}_{\alpha \in B_{\ell_{\infty}}} \subset \mathcal{M}_0(B_{Y})$ and therefore $\{\tau_{\alpha}\circ R\}_{\alpha \in B_{\ell_{\infty}}} \subset \mathcal{M}_0(B_{C(K)})$, where $R$ is the restriction mapping $R:H^{\infty}(B_{C(K)})\to H^{\infty}(B_Y)$, which is clearly a homomorphism. Note that the characters $\{\tau_{\alpha}\circ R\}_{\alpha \in B_{\ell_{\infty}}}$ are all distinct due to Theorem 1.1 in \cite{AB}, because $\ell_{\infty}$ is an isometrically injective space (Proposition 2.5.2 in \cite{AK}), so there exists a norm-one linear map $S:C(K) \to\ell_{\infty}$ such that $S|_{c_0}=I_{c_0}.$

\bigskip

We prove in Corollary \ref{co3.4} that if the latter cluster value problem has an affirmative answer at some point of $B_{C(K)}$, then it has an affirmative answer at all points of $B_{C(K)}.$ For that let us first establish the following lemmas, the first of which is a folklore result mentioned e.g. in \cite{V} and \cite{BKU}, but since there seems to be no proof in the literature we will sketch the proof.

\begin{lm}\label{lm3.2}
Let $f_0 \in B=B_{C(K)}.$ $T: B \to B$ given by $$T(f)=\frac{f-f_0}{1-\overline{f_0}\cdot f} \; \; \; \forall f \in B,$$ is biholomorphic.
\end{lm}
\begin{proof}

Set $\delta_0=||f_0||.$ 

\smallskip

Let us start by showing that $T$ is well defined, i.e. $||Tf||<1$ when $||f||<1.$

\smallskip

Let $f \in B.$ We can find $\delta \in (\delta_0,1)$ such that $||f||\leq \delta.$

\smallskip

For every $t_0 \in K,$ let $z=f(t_0)$ and $c=f_0(t_0),$ so that $T(f)(t_0)=\frac{z-c}{1-\overline{c} z}.$

\smallskip

Let $\Delta$ denote the open unit disk in the complex plane $\mathbb{C}.$

\smallskip

Since $\sigma: (\delta\cdot\overline{\Delta})\times(\delta_0\cdot\overline{\Delta}) \to \Delta$ given by $\sigma(z,c)=\frac{z-c}{1-\overline{c}z}$ is continuous, then $\sigma((\delta\cdot\overline{\Delta})\times(\delta_0\cdot\overline{\Delta}))$ is a compact subset of $\Delta,$ so there exists $\delta_1<1$ so that $\sigma((\delta\cdot\overline{\Delta})\times(\delta_0\cdot\overline{\Delta}))\subset \delta_1\overline{\Delta}.$

\smallskip

Thus $||Tf||\leq \delta_1 <1.$

\bigskip

Let us now show that $T$ is also holomorphic, or equivalently, $\mathbb{C}$-differentiable. For $f \in B$ fixed, the linear mapping $L:C(K) \to C(K)$ given by $L(h)=\frac{1-|f_0|^2}{(1-\overline{f_0}f)^2}h$ satisfies that, for $h\neq 0$ small enough,

\begin{align*}
\frac{T(f+h)-T(f)-L(h)}{||h||}
&=(\frac{f+h-f_0}{1-\overline{f_0}(f+h)}-\frac{f-f_0}{1-\overline{f_0}f}-\frac{1-|f_0|^2}{(1-\overline{f_0}f)^2}h)/||h||\\
&=(\frac{1-|f_0|^2}{1-\overline{f_0}f}\cdot\frac{h}{1-\overline{f_0}(f+h)}-\frac{1-|f_0|^2}{(1-\overline{f_0}f)^2}h)/||h||\\
&=\frac{\overline{f_0}h}{(1-\overline{f_0}f)^2(1-\overline{f_0}(f+h))}(1-|f_0|^2)h/||h||,
\end{align*}

which goes to zero as $h \to 0.$ Thus $T$ is holomorphic.

\medskip

\bigskip

Since $T$ clearly has a necessarily holomorphic inverse ($S(f)=\frac{f+f_0}{1+\overline{f_0}\cdot f}$), we have that $T$ is a biholomorphic function on $B$ that sends $f_0$ to the function identically zero.

\end{proof}

\begin{lm}\label{lm3.3}
The biholomorphic function $T$ from the previous lemma induces a mapping $\hat{T}$ on the spectrum $M_{H(B)}$, where $H$ denotes either the algebra $A_u$ or the algebra $H^{\infty},$ that maps $\mathcal{M}_{f_0}(B)$ onto $\mathcal{M}_0(B).$
\end{lm}
\begin{proof}

Note that $T$ is a Lipschitz function. Indeed, if $f, g \in B,$
$$||T(f)-T(g)||=||\frac{(1-|f_0|^2)(f-g)}{(1-\overline{f_0}f)(1-\overline{f_0}g)}|| \leq \frac{1}{(1-||f_0||)^2} ||f-g||.$$
Thus for every $\psi \in H(B),$ $\psi \circ T \in H(B)$. So  $\hat{T}: M_{H(B)} \to M_{H(B)},$ given by
$$\hat{T}(\tau)(\psi)=\tau(\psi \circ T), \; \; \forall \tau \in M_{H(B)}, \; \psi \in H(B), $$
is well defined. Moreover, given $\tau \in \mathcal{M}_{f_0}(B)$ and $\psi \in A_u(B),$
$$\hat{T}(\tau)(\psi)=\tau(\psi \circ T)=(\psi \circ T)(f_0)=\psi(0),$$
i.e. $\hat{T}(\tau) \in \mathcal{M}_0(B),$ for every $\tau \in \mathcal{M}_{f_0}(B).$

\smallskip

Now, given $\tau \in \mathcal{M}_0(B)$ it is clear that $\hat{\tau}: H(B) \to \mathbb{C}$ given by 
$$\hat{\tau}(\psi)=\tau(\psi \circ T^{-1}), \; \; \forall \psi \in H(B),$$
is in $M_{H(B)},$ actually in $\mathcal{M}_{f_0}(B),$ and $ \forall \; \psi \in H(B),$
$$\hat{T}(\hat{\tau})(\psi)=\hat{\tau}(\psi \circ T)=\tau(\psi),$$
i.e. $\hat{T}(\hat{\tau})=\tau.$
\end{proof}

The reader can easily check that the previous mapping $\hat{T}$ is actually a homeomorphism.

\begin{co}\label{co3.4}
For $X=C(K)$, the cluster value theorem of $H^{\infty}(B)$ over $A_u(B)$ at $0$ is equivalent to the cluster value theorem of $H^{\infty}(B)$ over $A_u(B)$ at every $f_0 \in B$.
\end{co}
\begin{proof}
Let $f_0 \in B$ and set $T$ as in Lemma \ref{lm3.2}. Then, $\forall \; \psi \in H^{\infty}(B),$
\begin{align*}
\hat{\psi}(\mathcal{M}_{0}(B))&=\hat{\psi} \circ \hat {T} (\mathcal{M}_{f_0}(B))=\widehat{\psi \circ T} (\mathcal{M}_{f_0}(B)),\\
\mathsf{Cl}_{B}(\psi, 0)&=\mathsf{Cl}_{B}(\psi \circ T, f_0),
\end{align*}
because $\psi \circ T \in H^{\infty}(B)$ too, and $T^{-1}(f)=(f+f_0)\sum_{n=0}^{\infty}(-\overline{f_0}f)^n \; \; \forall f \in B_{C(K)}$ is polynomially-star continuous, because sums and norm limits of polynomially-star continuous maps are polynomially-star continuous, as well as multiplication by a fixed element of $C(K)$. 
\end{proof}

%Let us note that, from the Gelfand Representation Theorem, the previous result holds for any nonzero unital commutative $C^*$-algebra.  

We now argue that we can extend the previous conclusions to the open unit ball of the second dual of $C(K):$ 

\medskip

In the statement of Lemma \ref{lm3.2}, we can rewrite $\frac{f-f_0}{1-\overline{f_0}\cdot f}$ as $(f-f_0)\sum_{n=0}^{\infty} (\overline{f_0}f)^n.$ Since it is known that $C(K)^{**}$ is a commutative $C^*-$algebra that extends the $C^*$ structure of $C(K)$ (see pp.~310-311 in \cite{Du} and p.~43 in \cite{S}), then Lemma \ref{lm3.2} extends in the following manner.

\begin{lm}\label{lm3.5}
Given $f_0^{**} \in B_{C(K)^{**}},$ let $T_{f_0^{**}}: B_{C(K)^{**}} \to B_{C(K)^{**}}$ be given by $$T_{f_0^{**}}(f^{**})=(f^{**}-f_0^{**})\sum_{n=0}^{\infty} (\overline{f_0^{**}}f^{**})^n \; \; \; \forall f^{**} \in B_{C(K)^{**}}.$$ Then $T_{f_0^{**}}$ is biholomorphic.
\end{lm}

Similarly, we clearly obtain the following analogues of Lemma \ref{lm3.3} and Corollary \ref{co3.4}.

\begin{lm}\label{lm3.6}
For each $f_0^{**} \in B_{C(K)^{**}},$ the biholomorphic function $T_{f_0^{**}}$ from the previous lemma induces a mapping $\widehat{T_{f_0^{**}}}$ on the spectrum $M_{H(B^{**})}$, where $H$ denotes either the algebra $A_u$ or the algebra $H^{\infty},$ that maps $\mathcal{M}_{f_0^{**}}(B^{**})$ onto $\mathcal{M}_0(B^{**})$.
\end{lm}

\begin{co}\label{co3.7}
For $X=C(K)$, the cluster value theorem of $H^{\infty}(B^{**})$ over $A_u(B^{**})$ at $0$ is equivalent to the cluster value theorem of $H^{\infty}(B^{**})$ over $A_u(B^{**})$ at every $f_0^{**} \in B_{C(K)^{**}}$.
\end{co}

Note that this last result is actually a consequence of Corollary \ref{co3.4} since the double dual of a space of continuous functions is again a space of continuous functions.

\bigskip

\begin{center}
\textbf{Acknowledgement}
\end{center}

\bigskip

The authors thank Richard Aron and Manuel Maestre for their communications.

\end{document}